\documentclass{amsart}
\usepackage{nameref}
\usepackage[english]{babel}	
\usepackage[utf8]{inputenc}		
\usepackage[T1]{fontenc}

\newcommand{\citep}[1]{\cite{#1}}
\usepackage{graphicx}
\usepackage{graphics}
\usepackage{epsfig}
\usepackage{amsmath,amsfonts,amsthm,amssymb,latexsym,fge}
\usepackage{mathrsfs,bbm,libertine, microtype}%,mathabx}%,bm,dsfont,}
\usepackage[scr]{rsfso}			% police script

\makeatletter

\newtheoremstyle{definition}
   {8pt}{5pt}{        }{0pt}{\scshape}{.}{5pt}{\thmname{#1}\thmnote{\;\textrm{[#3]}}}
\newtheoremstyle{theorems}
   {5pt}{3pt}{\itshape}{0pt}{\bfseries\scshape}{\textbf{.}}{5pt}{}%{\thmname{#1}\thmnumber{~#2}
\newtheoremstyle{ptheorems}
   {5pt}{3pt}{\itshape}{0pt}{\bfseries\scshape}{\textbf{.}}{5pt}{\thmname{#1}}
\makeatother

\theoremstyle{ptheorems}
   \newtheorem*{pthm}{Theorem}
\theoremstyle{theorems}
   \newtheorem{thm}{Theorem}[section]
   
   \newtheorem{lemma}	[thm]{Lemma}
   \newtheorem{proposition}[thm]{Proposition}
   \newtheorem{corollary}	[thm]{Corollary}
   
\makeatletter
\@addtoreset{thm}{section}
\makeatother
\theoremstyle{remark}

\theoremstyle{definition}

\newcommand{\dfn}[1]{\emph{\textbf{#1}}\index{#1}}

\def\emb#1{{\ensuremath{#1}}}			
\def\ss#1{\mathscr{#1}}				
\def\dd#1{\mathbbm{#1}}

			\def\ZZ{\dd Z}			
	      \def\AA{\dd A}
\def\set #1{\emb{\left\lbrace \:	 #1 \:\right\rbrace }}

\def\setm #1 | #2{\set{#1 \; \mskip1mu\vrule\mskip1mu \; #2}}
\def\lst #1:#2{\emb{ {#1},\,\dotsc,\,{#2}}}	
\def\seq #1(#2:#3){\lst {#1}_{#2}:{#1}_{#3}}
\def\som #1(#2:#3){\emb{ {#1}_{#2}+\,\dotsc,\,+ {#1}_{#3} }}
\def\prd #1(#2:#3){\emb{ {#1}_{#2}\dotsc{#1}_{#3} }}		
\def\fnarg #1:#2->#3.{\emb{#1:\arr{#2}->{#3}}}

\newcommand{\cat}[1][C]{\emb{\ss{#1}}}		\newcommand{\der}{\emb{\cat[D]}}
\newcommand{\kk}    {{\mathsf{k}}}			\newcommand{\mkQ}{\emb{\mod{\kk Q}}}

\newcommand{\id}    {{\rm id}}					

\newcommand{\Hom}   {\operatorname{Hom}}    \newcommand{\End}   {\operatorname{End}}
    
    \newcommand{\add}   {\operatorname{add}}

	  \newcommand{\HH}   {\operatorname{H}}
\newcommand{\Ima}   {\operatorname{Im}}     
\renewcommand{\mod}	{\operatorname{mod}}	  \newcommand{\rk}    {\operatorname{rk}}
    
\newcommand{\Ext}   {\operatorname{Ext}}

\newcommand{\Succ}    {{\rm Succ}\,}

				 \renewcommand{\ge}		{\geqslant}
\newcommand{\torsion}[2]{\left(\mathscr #1,\mathscr #2 \right)}

\newcommand{\sset}{\subseteq}
\renewcommand{\setminus}{\mathbin{\fgebackslash}\,}
\begin{document}
\title{Split $t$-structures and torsion pairs in hereditary categories}

   %% First author (Note: The order of the items here is important!)
   \author[I. Assem]{Ibrahim Assem}
   \address{D\'epartement de math\'ematiques, Facult\'e des sciences, Universit\'e de Sherbrooke,
       Sherbrooke, Qu\'ebec J1K 2R1, Canada.}
{\email{ibrahim.assem@usherbrooke.ca}

   %% Second author (Note: The order of the items here is important!)

   \author[M. J. Souto-Salorio]{Mar\'ia Jos\'e Souto-Salorio}
\address{Departamento de Computaci\'on, Facultade de Inform\'atica, Universidade da Coru\~{n}a,
Campus de A Coru\~{n}a, 15071 A Coru\~{n}a, Espa\~{n}a}
 {\email{maria.souto.salorio@udc.es}

   %% Second author (Note: The order of the items here is important!)
   \author[S. Trepode]{Sonia Trepode}
    \address{Departamento de Matem\'atica, Facultad de Ciencias Exactas y Naturales, Funes 3350,
       Universidad Nacional de Mar del Plata, CONICET,  7600 Mar del Plata, Argentina.}
 {\email{strepode@mdp.edu.ar}

   %% AMS subject classification; see http://www.ams.org/msc
   %% Only one Primary. Possibly several Secondary.

   %% Keywords and phrases

   \begin{abstract}
      We give necessary and sufficient conditions for torsion pairs in a hereditary category to be in bijection with $t$-structures in the bounded derived category of that hereditary category. We prove that the existence of a split $t$-structure with nontrivial heart in a semiconnected Krull-Schmidt category implies that this category is equivalent to the derived category of a hereditary category. We construct a bijection between split torsion pairs in the module category of a tilted algebra having a complete slice in the preinjective component with corresponding $t$-structures. Finally, we classify split $t$-structures in the derived category of a hereditary algebra.
   \end{abstract}

   \maketitle

   \section*{Introduction}
   The notion of torsion pair, or torsion theory, in an abelian category was introduced by S. Dickson in the 1960's, see~\citep{D}.
   Modeled after properties of torsion and torsion-free abelian groups, it gives information on the morphisms in the category.
   The analogous concept in a triangulated category is that of $t$-structure, introduced by Be{\u\i}linson, Bernstein and Deligne in~\citep{BBD}.

   The objective of the present paper is to compare torsion pairs in a hereditary category $\cat[H]$ and $t$-structures in the bounded derived category $\der^b(\cat[H])$ with special attention to those which are split.
   Let $\kk$ be an algebraically closed field.
   Following~\citep{HRS1}, we say that a connected abelian $\kk$-category $\cat[H]$ is hereditary whenever the bifunctor $\Ext^2_{\cat[H]}$ vanishes and the category has finite dimensional $\Hom$ and $\Ext^1$-spaces.

   Our starting point is an observation in~\citep{HRS1} saying that a torsion pair in $\cat[H]$ lifts to a $t$-structure in $\der^b(\cat[H])$.
   It is easy to see that the reverse procedure is obtained by taking the trace of the $t$-structure on $\cat[H]$.
   We deduce a bijective correspondence between torsion pairs in $\cat[H]$ and $t$-structures $(\cat[U],\cat[V])$ such that $\cat[H][1]\subseteq \cat[U]$ and
   $\cat[H]\subseteq \cat[V]$. Here, $[\,\cdot\,]$ denotes the shift of the derived category $\der^b(\cat[H])$.

   We then specialise our study to the split torsion pairs, namely those for which every indecomposable object is either torsion or torsion-free.
   We wish to study when they lift to split $t$-structures, that is, to $t$-structures $(\cat[U],\cat[V])$ for which every indecomposable object belongs either to $\cat[U]$ or to $\cat[V][-1]$.
   Our first result says that the mere existence of a split $t$-structure with nontrivial heart in a semiconnected Krull-Schmidt $\kk$-category implies that this category is equivalent to the derived category of a hereditary category.
   This generalises~\citep{BR}(4.2).
   We next look at tilted algebras.
   Let $H$ be an hereditary algebra.
   We recall that an algebra $A$ is called tilted of type $H$ if there exists a tilting $H$-module $T$ such that $A=\End T$.
   Tilted algebras are characterised by the existence of complete slices in their Auslander-Reiten quivers, see~\citep{ASS}.
   Denoting by $\cat_1$ the transjective component of the Auslander-Reiten quiver of $\der^b(\mod H)$ obtained by gluing the preinjective component of $H$ with the first shift of the postprojective component, we prove the following theorem.

   \begin{pthm}
      Let $A$ be a representation-infinite tilted algebra of type $H$ having a complete slice in the preinjective component. Then there exist bijective correspondences between:
      \begin{itemize}
         \item [(a)] Split torsion pairs $(\cat[T],\cat[F])$ in $\mod A$ with all preinjectives in $\cat[T]$ and all postprojectives in $\cat[F]$.
             
         \item [(b)] Split torsion  pairs $(\cat[T]',\cat[F]')$ in $\mod H$ with all preinjectives in $\cat[T]'$ and all postprojectives in $\cat[F]'$.
             
         \item [(c)]Split $t$-structures $(\cat[U],\cat[U]^\perp[1])$ in $\der^b(\mod H)$ with $\cat_1$ lying in the heart.
      \end{itemize}
   \end{pthm}

   While the bijection between (b) and (c) is constructed categorically using the description of the derived category and does not require the splitting hypothesis, the bijection between (a) and (b) requires the use of the tilting functors and uses essentially that the torsion pairs are split.

   Finally, we complete our results by deriving a classification of the split $t$-structures in the derived category of a hereditary algebra.

   \section{Preliminaries}
   \subsection{Notation}
   Throughout this paper, $\kk$ denotes a fixed algebraically closed field.
   All our algebras are finite dimensional $\kk$-algebras and our modules are finitely generated right modules.
   The module category of an algebra $A$ is denoted by $\mod A$.
   All our categories are additive Krull-Schmidt $\kk$-categories.
   If $\cat$ is a category and $\cat[D]$ a full subcategory of $\cat$, we write $X\in\der$ to express that $X$ is an object in $\der$.
   The \dfn{right} and \dfn{left orthogonals} of $\der$ are the full subcategories of $\cat$ defined respectively by their object classes as:
   \begin{align*}
   \der^\perp   &= \set{Y\in\cat \mid \Hom_{\cat }(X,Y)=0 \text{ for all }X\in\der}\text{, and}\\
   {^\perp}\der &= \set{Y\in\cat \mid \Hom_{\cat }(Y,X)=0 \text{ for all }X\in\der}.
   \end{align*}

   Given two full subcategories $\der_1$, $\der_2$ of $\cat$ such that $\Hom_{\cat }(X_2,X_1)=0$ for all $X_1\in\der_1$ and $X_2\in\der_2$, then we denote by $\der_1\vee\der_2$ the full subcategory of $\cat$ generated by all objects of $\der_1$ and $\der_2$. For all basic notions of representation theory, we refer the reader to~\citep{ARS,ASS}.

   \subsection{Torsion pairs in hereditary categories}
   A connected abelian $\kk$-category $\cat[H]$ is \dfn{hereditary} if, for all $X$, $Y\in\cat[H]$, we have $\Ext^2_{\cat[H]}(X,Y)=0$ while $\Hom_{\cat[H]}(X,Y)$ and $\Ext^1_{\cat[H]}(X,Y)$ are finite dimensional $\kk$-vector spaces.

   An object $T$ in a hereditary category $\cat[H]$ is \dfn{tilting} if
   $\Ext^1_{\cat[H]}(T,T)=0$ and if \linebreak[4]$\Hom_{\cat[H]}(T,X)=0=\Ext^1_{\cat[H]}(T,X)$ imply $X=0$.

   It is shown in~\citep{H2} that, if $\cat[H]$ is a hereditary category with tilting object, then $\cat[H]$ is derived-equivalent to $\mod H$ for some hereditary algebra $H$, or to $\mod C$, for some canonical algebra $C$, in the sense of~\citep{R1}.
   In each of these two cases, the bounded derived category $\der^b(\cat[H])$ is a triangulated category with Serre duality.

   A \dfn{torsion pair} $(\cat[T],\cat[F])$ in $\cat[H]$ is a pair $(\cat[T],\cat[F])$ of full subcategories such that:
   \begin{itemize}
      \item [(a)]For all $X\in\cat[T]$, $Y\in\cat[F]$, we have $\Hom_{\cat[H]}(X,Y)=0$.
      
      \item [(b)]For any $Y\in\cat[H]$, there exists a short exact sequence (the \dfn{canonical sequence} of $Y$) of the form $0 \longrightarrow X \longrightarrow Y \longrightarrow Z \longrightarrow 0 $ with $X\in\cat[T]$ and $Z\in\cat[F]$.
   \end{itemize}

   Objects in $\cat[T]$ are called \dfn{torsion}, while those in $\cat[F]$ are called \dfn{torsion-free}.

   Equivalently, a pair $(\cat[T],\cat[F])$ of full subcategories is a torsion pair if and only if $\cat[T]={^\perp}\cat[F]$, or if and only if $\cat[F]=\cat[T]^\perp$.
   For instance, any tilting object $T$ in $\cat[H]$ induces a torsion pair $(\cat[T](T),\cat[F](T))$ where
   $\cat[T](T)=\set{X\in\cat[H]\mid\Ext^1_{\cat[H]}(T,X)=0}$ and
   $\cat[F](T)=\set{Y\in\cat[H]\mid\Hom_{\cat[H]}(T,Y)=0}$, see~\citep{HRS1,HR}.

   A torsion pair $(\cat[T],\cat[F])$ is \dfn{split} if every indecomposable object in $\cat[H]$ belongs either to $\cat[T]$ or to $\cat[F]$.

   \subsection{$t$-structures in triangulated categories}

   Let $\cat$ be a triangulated category with shift $[\,\cdot\,]$.
   All triangles considered will be distinguished triangles.
   A full subcategory $\cat[U]$ of $\cat$, closed under direct summands, is \dfn{suspended} if it is closed under positive shifts and extensions, that is,
   \begin{itemize}
      \item If $X\in\cat[U]$, then $X[1]\in\cat[U]$.
      
      \item If $X \longrightarrow Y \longrightarrow Z \longrightarrow X[1]$ is a triangle in $\cat$, with $X,Z\in\cat[U]$, then
      $Y\in\cat[U]$.
   \end{itemize}

   Dually, one defines \dfn{cosuspended subcategories}.

   A \dfn{$t$-structure}, see~\citep{BBD}(1.3.1), is a pair $(\cat[U],\cat[V])$ of full subcategories of $\cat$ such that
   \begin{itemize}
      \item If $X\in\cat[U]$ and $Y\in\cat[V][-1]$, then $\Hom_{\cat}(X,Y)=0$.
      
      \item $\cat[U]\subseteq\cat[U][-1]$ and $\cat[V]\supseteq\cat[V][-1]$.
      
      \item For any $Y\in\cat$, there exists a triangle $X \longrightarrow Y \longrightarrow Z \longrightarrow X[1]$ in $\cat$ with $X\in\cat[U]$, $Z\in\cat[V][-1]$.
   \end{itemize}

   The $t$-structure $\torsion UV$ is \dfn{split} if every indecomposable
   object in $\cat$ belongs either to $\cat[U]$ or to $\cat[V][-1]$.

   A suspended subcategory $\cat[U]$ of $\cat$ is an \dfn{aisle} if it is contravariantly finite in $\cat$.
   It is proved in~\citep{KV}(1.1)(1.3) that the following conditions are equivalent for a suspended subcategory $\cat[U]$ of $\cat$:
   \begin{itemize}
      \item $\cat[U]$ is an aisle.
      
      \item $(\cat[U],\cat[U]^\perp[1])$ is a $t$-structure.
      
      \item For any $Y\in\cat$, there exists a triangle $X \longrightarrow Y \longrightarrow Z \longrightarrow X[1]$ in $\cat$ with $X\in\cat[U]$, $Z\in\cat[U]^\perp$.
   \end{itemize}

   The dual notion is that of \dfn{coaisle}, for which the dual statement holds.

   The \dfn{heart} of the $t$-structure $(\cat[U],\cat[U]^\perp[1])$ is the full subcategory $\cat[U]\cap\cat[U]^\perp[1]$, which is abelian, because of~\citep{BBD}(1.3.6).

   Given a full subcategory $\cat[U]$ of $\cat$ closed under extensions, an object $X\in\cat[U]$ is \dfn{Ext-projective} in $\cat[U]$ if $\Hom_{\cat}(X,Y[1])=0$ for all $Y\in\cat[U]$, see~\citep{AS}.
   If $\cat$ has Serre duality, then an indecomposable object $X\in\cat[U]$ is Ext-projective in $\cat[U]$ if and only if $\tau X\in\cat[U]^\perp$, see~\citep{AST}(1.5).
   The dual notion is that of \dfn{Ext-injective} in $\cat[U]$, for which the dual statement holds.

   \section{The lift and trace maps}

   Let $\cat[H]$ be a hereditary category.
   In this section, we compare torsion pairs in $\cat[H]$ and $t$-structures in $\der^b(\cat[H])$ by means of two maps.
   We start by recalling the following lemma~\citep{HRS1}(I.2.1).

   \begin{lemma}
      A torsion pair $(\cat[T],\cat[F])$ in an abelian category $\cat[A]$ induces a $t$-structure \linebreak[4]$(\cat[U]_{\cat[T]},\cat[U]^\perp_{\cat[T]}[1])$ in $\der^b(\cat[A])$ by:
      \begin{align*}
      \cat[U]_{\cat[T]} &=\set{X\in\der^b(\cat[A])\mid \HH^i(X)=0\text{~~for all~}i>0,~ \HH^0(X)\in\cat[T]}\text{, and}\\
      \cat[U]_{\cat[T]}^\perp &=\set{X\in\der^b(\cat[A])\mid \HH^i(X)=0\text{~~for all~}i<-1,~ \HH^{-1}(X)\in\cat[F]}.\qed
      \end{align*}
   \end{lemma}

   Thus, there exists a map %$\begin{tikzcd}
   $\phi:(\cat[T],\cat[F])\longrightarrow (\cat[U]_{\cat[T]},\cat[U]_{\cat[T]}^\perp[1])$
   %\end{tikzcd 
   from the class of torsion pairs in $\cat[H]$ to the class of $t$-structures in $\der^b(\cat[A])$.
   The map $\phi$ is called the \dfn{lift map}.
   We now proceed to define a partial inverse map.

   \begin{lemma}\label{l:2.2}
      Let $\cat[U]$ be an aisle, and $\cat[V]$ a coaisle in $\der^b(\cat[H])$.
      \begin{itemize}
         \item If $\cat[H][1]\subseteq \cat[U]$, then $\cat[T]=\cat[U]\cap\cat[H]$ is a torsion class in $\cat[H]$.
         
         \item If $\cat[H][-1]\subseteq \cat[V]$, then $\cat[F]=\cat[V]\cap\cat[H]$ is a torsion-free class in $\cat[H]$.
      \end{itemize}
   \end{lemma}
   \begin{proof}
      We only prove (a), because the proof of (b) is dual.
      In order to prove that $\cat[T]=\cat[U]\cap\cat[H]$ is a torsion class, it suffices to prove that $\cat[T]={^\perp}(\cat[T]^\perp)$.
      Trivially, one has $\cat[T]\subseteq {^\perp}(\cat[T]^\perp)$.
      Conversely, assume $Y\in{^\perp}(\cat[T]^\perp)$ that is, $\Hom_{\cat[H]}(Y,-)\big|_{\cat[T]^\perp}=0$.
      Because $\cat[U]$ is an aisle, there exists a triangle
      $X \longrightarrow Y \longrightarrow Z \longrightarrow X[1]$
      with $X\in\cat[U]$, $Z\in\cat[U]^\perp$.
      Let $Z'$ be an indecomposable summand of $Z$.
      Because $Y\in\cat[H]$, then $Z'$ is concentrated in degree $0$ or $1$.
      Assume $Z'=M[1]$ for some $M\in\cat[H]$.
      The hypothesis yields $Z'\in\cat[U]$.
      Hence $Z'\in\cat[U]\cap\cat[U]^\perp=0$, a contradiction.
      Therefore, all indecomposable summands of $Z$ are concentrated in degree $0$, that is $Z\in\cat[H]$.
      Our hypothesis on $Y$ yields $g=0$, so $Y$ is a direct summand of $X$.
      Thus $Y\in\cat[U]$.
      Because $Y\in\cat[H]$, then $Y\in\cat[T]$. This completes the proof.
   \end{proof}

   \begin{corollary}\label{c:2.3}
      Let $(\cat[U],\cat[U]^\perp[1])$ be a $t$-structure in $\der^b(\cat[H])$ such that $\cat[H][1]\subseteq \cat[U]$ and $\cat[H][-1]\subseteq \cat[U]^\perp$.
      Then $(\cat[U]\cap\cat[H], \cat[U]^\perp \cap\cat[H])$ is a torsion pair in $\cat[H]$.
   \end{corollary}
   \begin{proof}
      Because of lemma\ref{l:2.2}, $\cat[U]\cap\cat[H]$ is a torsion class, and $\cat[U]^\perp\cap\cat[H]$ is a torsion-free class.
      There remains to show that $(\cat[U]\cap\cat[H])^\perp = \cat[U]^\perp\cap\cat[H]$.
      Clearly, $\cat[U]^\perp\cap\cat[H]\subseteq (\cat[U]\cap\cat[H])^\perp$.
      Conversely, let $Y\in (\cat[U]\cap\cat[H])^\perp$.
      There exists a triangle in $\der^b(\cat[H])$
      $X \longrightarrow Y \longrightarrow Z \longrightarrow X[1]$
      with $X\in\cat[U]$, $Z\in\cat[U]^\perp$.
      Because $Y\in\cat[H]$, every indecomposable summand $X'$ of $X$ is concentrated in degree $0$ or $-1$.
      However, $\cat[H][-1]\subseteq\cat[U]^\perp$ thus, if $X'$ is concentrated in degree $-1$, then $X'\in\cat[U]\cap\cat[U]^\perp=0$, a contradiction.
      Hence, $X$ is concentrated in degree $0$, that is, $X\in\cat[H]$.
      So $X\in\cat[U]\cap\cat[H]$.
      Similarly, $Z\in\cat[U]^\perp\cap\cat[H]$.
      But then $Y\in (\cat[U]\cap\cat[H])^\perp$ implies $f=0$, so that $Y\in\cat[U]^\perp\cap\cat[H]$.
   \end{proof}

   Thus, we have a map %$\begin{tikzcd}
   $\psi:(\cat[U],\cat[U]^\perp[1])\longrightarrow (\cat[U]\cap\cat[H], \cat[U]^\perp \cap\cat[H])$
   
   %\end{tikzcd}
   from the class of $t$-structures in $\der^b(\cat[H])$ with $\cat[H][1]\subseteq \cat[U]$, $\cat[H][-1]\subseteq \cat[U]^\perp$ to the class of torsion pairs in $\cat[H]$.
   The map $\psi$ is called the \dfn{trace map}.
   We now prove that the trace map and lift map are inverse to each other.

   \begin{proposition}\label{p:2.4}
      Let $\cat[H]$ be an hereditary category, then the lift and the trace maps are inverse bijections between the class of all torsion pairs $(\cat[T],\cat[F])$ in $\cat[H]$ and the class of all $t$-structures $(\cat[U],\cat[U]^\perp[1])$ in $\der^b(\cat[H])$ such that $\cat[H][1]\subseteq \cat[U]$ and $\cat[H][-1]\subseteq \cat[U]^\perp$.
   \end{proposition}
   \begin{proof}
      We first show that the image of $\phi$ lies in the class of $t$-structures satisfying the stated conditions.
      Indeed, let $(\cat[T],\cat[F])$ be a torsion pair in $\cat[H]$, and $\phi(\cat[T],\cat[F])=(\cat[U]_{\cat[T]},\cat[U]_{\cat[T]}^\perp[1])$.
      Using the description of the category $\der^b(\cat[H])$, see~\citep{H1}, this may be expressed as follows
      \begin{align*}
      \cat[U]_{\cat[T]}&=\cat[T]\vee \left(\bigvee_{j>0}\cat[H][j]\right)\text{, and}\\
      \cat[U]_{\cat[T]}^\perp&=\left(\bigvee_{j<0}\cat[H][j]\right)\vee\cat[F]
      \end{align*}
      where $\cat[T]$ and $\cat[F]$ are considered as embedded in $\cat[H][0]\subseteq \der^b(\cat[H])$.
      In particular, $\cat[H][1]\subseteq \cat[U]_{\cat[T]}$ and $\cat[H][-1]\subseteq\cat[U]_{\cat[T]}^\perp$.

      We now prove that $\phi$ and $\psi$ are inverse bijections.
      If $(\cat[T],\cat[F])$ is a torsion pair in $\cat[H]$, then it follows immediately from the definitions that $\psi\phi(\cat[T],\cat[F])=(\cat[T],\cat[F])$.
      Conversely, let $(\cat[U],\cat[U]^\perp[1])$ be a $t$-structure in $\der^b(\cat[H])$ such that $\cat[H][1]\subseteq\cat[U]$ and $\cat[H][-1]\subseteq\cat[U]^\perp$, then $\phi\psi(\cat[U],\cat[U]^\perp[1])=(\cat[U]_{\cat[U]\cap\cat[H]},\cat[U]_{\cat[U]\cap\cat[H]}^\perp[1])$ where
      \begin{align*}
      \cat[U]_{\cat[U]\cap\cat[H]}&=(\cat[U]\cap\cat[H])\vee\left(\bigvee_{j>0}\cat[H][j]\right)\subseteq \cat[U]\text{, and}	\\
      \cat[U]_{\cat[U]\cap\cat[H]}^\perp[1]&=\left(\bigvee_{j<0}\cat[H][j]\right)\vee (\cat[U]^\perp\cap\cat[H])\subseteq \cat[U]^\perp
      \end{align*}
      where, again, $\cat[H]$ is identified with $\cat[H][0]$.
      But then $\cat[U]={^\perp}(\cat[U]^\perp)\subseteq
      {^\perp}(\cat[U]^\perp_{\cat[U]\cap\cat[H]})=\cat[U]_{\cat[U]\cap\cat[H]}$.
      We conclude that $\cat[U]=\cat[U]_{\cat[U]\cap\cat[H]}$ and therefore
      $\phi\psi(\cat[U],\cat[U]^\perp[1])=(\cat[U],\cat[U]^\perp[1])$ as required.
   \end{proof}

   \begin{corollary}
      Let $(\cat[T],\cat[F])$ be a torsion pair in an hereditary category $\cat[H]$. Then the indecomposable objects in $\der^b(\cat[H])$ lying in $\cat[T]\vee\cat[F][1]$ are exactly the indecomposables in the heart $\cat[U]_{\cat[T]}\cap\cat[U]_{\cat[T]}[1]$ of the lifted $t$-structure.
   \end{corollary}
   \begin{proof}
      The indecomposable objects in the heart are concentrated in degrees $0$ and $1$, and therefore coincide with the indecomposable objects lying in $\cat[T]\vee\cat[F][1]$.
   \end{proof}

   \section{The case of hereditary algebras}

   In this section, we assume our hereditary category to be of the form $\cat[H]=\mod H$, where $H$ is a representation-infinite hereditary algebra.
   The representation theory of such an algebra $H$ is well-known, see, for instance, \citep{ARS,ASS}.
   Indecomposable $H$-modules are divided into three classes: $\cat[P]$, consisting of the postprojective modules, $\cat[R]$, consisting of the regular, and $\cat[I]$, consisting of the preinjective.
   Moreover, $\mod H=\cat[P]\vee\cat[R]\vee\cat[I]$ and any morphism from an object in $\cat[P]$ to one in $\cat[I]$ factors through the additive category $\add\cat[R]$ generated by $\cat[R]$.
   Also, the derived category $\der^b(\mod H)$ is described, for instance, in~\citep{H1}.
   Its indecomposable objects are also divided into classes: $\cat_j$, consisting of the transjective objects, and $\cat[R]_j$, of the regular ones, with $j$ running over $\ZZ$.
   These are related to $H$-modules as follows.
   We have $\cat_0=\cat[I][-1]\vee\cat[P]$ and, for each $j$, $\cat_j=\cat_0[j]$.
   Similarly,
   $\cat[R]_0=\cat[R]$ and $\cat[R]_j=\cat[R][j]$ for each $j$.
   We then have $\der^b(\mod H)=\bigvee_{j\in\ZZ}(\cat_j\vee\cat[R]_j)$ and any morphism from $\cat_j$ to $\cat_{j+1}$ factors through $\add\cat[R]_j$. The following picture (with morphisms going from left to right) may be helpful for the reader.
\vskip 1cm
\begin{figure}[h!]

%\begin{center}
 
\includegraphics[scale=0.98]{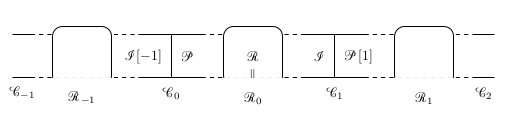}
 
%\end{center}
 
\end{figure}
 
   \begin{lemma}\label{l:3.1}
      Let $H$ be a representation-infinite hereditary algebra.
      \begin{itemize}
         \item [(a)] If $\cat[U]$ is an aisle in $\der^b(\mod H)$ such that $\cat_1\subseteq \cat[U]$, then $\bigvee_{j>0}(\cat_j\vee\cat[R]_j)\subseteq \cat[U]$.
             
         \item [(b)] If $\cat[V]$ is a coaisle in $\der^b(\mod H)$ such that $\cat_0\subseteq \cat[V]$, then $\bigvee_{j<0}(\cat_j\vee\cat[R]_j)\vee\cat_0\subseteq \cat[V]$.
      \end{itemize}
   \end{lemma}
   \begin{proof}
      We only prove (a), because the proof of (b) is dual. If $\cat_1\subseteq \cat[U]$, then, for each $j>0$, we have $\cat_j=\cat_1[j-1]\subseteq\cat[U]$.
      Now consider $\cat[R]_j$ for some $j>0$.
      If $Y\in\cat[R]_j$, there exists $X\in\cat_j$ such that $\Hom_{\der^b(\mod H)}(X,Y)\ne 0$.
      Because $\cat_j\subseteq\cat[U]$, we get $Y\notin\cat[U]^\perp$.
      In particular, $\cat[R]_j\cap\cat[U]^\perp=0$ for all $j\in\ZZ$.
      We now prove that $Y\in\cat[U]$.
      Consider the triangle
      $X \longrightarrow Y \longrightarrow Z \longrightarrow X[1]$
      with $X\in\cat[U]$, $Z\in\cat[U]^\perp$.
      Let $Z'$ be an indecomposable summand of $Z$.
      Then $Z'\notin\cat_t$, for any $t\ge j+1$, because $\cat_t\subseteq \cat[U]$.
      Therefore, $Z'\in\cat[R]_t$ for some $t\ge j$.
      However, $\cat[R]_t\cap\cat[U]^\perp=0$ for $t\ge j>0$, a contradiction.
      Therefore $g=0$ and so $Y$ is a direct summand of $X$.
      In particular, $X\in\cat[U]$.
   \end{proof}

   As a first corollary, we consider split torsion pairs and $t$-structures induced by sections.
   For sections in translation quivers, we refer the reader to~\citep{ASS} and recall that faithful sections are complete slices.
   We need the following notation. Let $\Sigma$ be a section in a translation quiver $\Gamma$.
   We denote by $\Succ\Sigma$ the set of all successors of $\Sigma$ in $\Gamma$, that is, of all $x$ in $\Gamma$ such that there exist $e$ in $\Sigma$ and a sequence of arrows %$\begin{tikzcd}[sep=small,cramped]
   $e=x_0 \longrightarrow x_1 \longrightarrow\cdots \longrightarrow x_t=x$
   %\end{tikzcd}$ 
   in $\Gamma$.

   \begin{corollary}\label{c:3.2}
      Let $H$ be a representation-infinite hereditary algebra.
      The lift and the trace maps restrict to inverse bijections between:
      \begin{itemize}
         \item [(a)] Split torsion pairs $(\cat[T],\cat[F])$ in $\mod H$ such that all indecomposable Ext-projectives in $\cat[T]$ form a section $\Sigma$ in $\Gamma(\mod H)$.
             
         \item [(a)]Split $t$-structures $(\cat[U],\cat[U]^\perp[1])$ in $\der^b(\mod H)$ such that all indecomposable \linebreak[4]Ext-projectives in $\cat[U]$ form a section $\Sigma$ in $\Gamma(\der^b(\mod H))$.
      \end{itemize}
   \end{corollary}
   \begin{proof}
      First, because of proposition~\ref{p:2.4} and their very definitions, the lift and the trace maps restrict to inverse bijections between split torsion pairs in $\mod H$ and split $t$-structures $(\cat[U],\cat[U]^\perp[1])$ in $\der^b(\mod H)$ such that $\mod H[1]\subseteq\cat[U]$ and $\mod H[-1]\subseteq\cat[U]^\perp$.
      Clearly, if $(\cat[T],\cat[F])$ is a split torsion pair as in (a), then $\cat[T]=\Succ\Sigma$ if $\Sigma$ is in $\cat[I]$, while $\cat[T]=\Succ\Sigma\vee\cat[R]\vee\cat[I]$ if $\Sigma$ lies in $\cat[P]$.
      Because of lemma~\ref{l:3.1} above, in the first case, it lifts to the $t$-structure $(\cat[U],\cat[U]^\perp[1])$ such that $\cat[U]=\Succ\Sigma\vee\cat[R]_1\vee\left(\bigvee_{j>1}(\cat_j\vee\cat[R]_j)\right)$, and in the second case, it lifts to the $t$-structure $(\cat[U],\cat[U]^\perp[1])$ such that $\cat[U]=\Succ\Sigma\vee\cat[R]_0\vee\left(\bigvee_{j>0}(\cat_j\vee\cat[R]_j)\right)$.
      Conversely, taking the trace of a $t$-structure of one of these two types in $\mod H$ yields a torsion pair of the required form.
   \end{proof}

   We are now able to state and prove the main result of this section.
   Observe that the two conditions $\cat_1\subseteq\cat[U]$ and $\cat_0\subseteq\cat[U]^\perp$ are equivalent to the sole condition $\cat_1\subseteq\cat[U]\cap\cat[U]^\perp[1]$, that is, $\cat_1$ is contained in the heart.

   \begin{thm}\label{t:3.3}
      Let $H$ be a representation-infinite hereditary algebra. The lift and trace maps restrict to inverse bijections between the class of all torsion pairs
      $\torsion TF$ in $\mod H$ such that $\cat[I]\subseteq\cat[T]$, $\cat[P]\subseteq\cat[F]$ and the class of all $t$-structures $\torsion U{U^\perp[1]}$ in $\der^b(\mod H)$ such that $\cat_1\subseteq \cat[U]\cap\cat[U]^\perp[1]$.
   \end{thm}
   \begin{proof}
      Let $\torsion U{U^\perp[1]}$ be a $t$-structure in $\der^b(\mod H)$ such that $\cat_1\sset\cat[U]$ and $\cat_0\sset\cat[U]^\perp$, and $\torsion TF$ is its trace, that is, $\cat[T]=\cat[U]\cap\mod H$ and $\cat[F]=\cat[U]^\perp\cap\mod H$.
      We claim that $\torsion TF$ is a torsion pair in $\mod H$ such that
      $\cat[I]\subseteq\cat[T]$, $\cat[P]\subseteq\cat[F]$.
      That $\torsion TF$ is a torsion pair follows from corollary~\ref{c:2.3}
      and the fact that, because of the hypothesis and lemma~\ref{l:3.1}, we have $\mod H[1]\sset\cat_1\vee\cat[R]_1\vee\cat_2\sset\cat[U]$ and $\mod H[-1]\sset \cat_{-1}\vee\cat[R]_{-1}\vee\cat_0\sset \cat[U]^\perp$.
      Moreover, $\cat[I]=\cat_1\cap\mod H\sset\cat[U]\cap\mod H=\cat[T]$, so that
      $\cat[I]\sset\cat[T]$.
      Similarly, $\cat[F]$ contains $\cat[P]=\cat_0\cap\mod H$.

      Conversely, let $\torsion TF$ be a torsion pair in $\mod H$ such that $\cat[I]\subseteq\cat[T]$, $\cat[P]\subseteq\cat[F]$, and let $\torsion {U_{\cat[T]}}{U_{\cat[T]}^\perp[1]}$ denote its lift to $\der^b(\mod H)$.
      We claim that $\cat_1\sset\cat[U]_{\cat[T]}$ and $\cat_0\sset\cat[U]_{\cat[T]}^\perp$.
      Let $X\in\cat_1$. If $X$ is an $H$-module, then $X\in\cat[I]\sset\cat[T]\sset\cat[U]_{\cat[T]}$.
      If not, then $X=M[1]$ for some $H$-module $M$.
      Taking cohomology, we get $\HH^{-1}(X)=M$ and $\HH^j(X)=0$ for all $j\ne -1$.
      In particular, $X\in\cat[U]_{\cat[T]}$.
      Similarly, $\cat_0\sset\cat[U]_{\cat[T]}^\perp$.

      Because of lemma~\ref{l:3.1}, we have
      $\bigvee_{j>0}\mod H[j]\sset\cat[U]_{\cat[T]}$ and
      $\bigvee_{j<0}\mod H[j]\sset\cat[U]_{\cat[T]}^\perp$.
      Then $\cat[U]_{\cat[T]}\cap\cat[U]_{\cat[T]}^\perp=0$ yields
      \begin{align*}
      \cat[U]_{\cat[T]}			&= \left(\cat[U]_{\cat[T]}\cap\mod H\right)\vee \left(\bigvee_{j>0}\mod H[j]\right)\\
      \cat[U]_{\cat[T]}^\perp &= \left(\bigvee_{j<0}\mod H[j]\right)\vee \left(\cat[U]_{\cat[T]}^\perp\cap\mod H\right).
      \end{align*}
      It is now clear that the lift and the trace maps are inverse bijections.
   \end{proof}

   For future reference, it is useful to observe that, because of their definitions, the lift and trace maps also restrict to inverse bijections between split torsion pairs and split $t$-structures satisfying the conditions of the theorem.

   Let $H$ be a wild hereditary algebra and $M$ a quasisimple module.
   Following~\citep{AK}, we define the \dfn{left cone} ($\longrightarrow M$) to be the full subcategory of $\mod H$ generated by all the indecomposable $H$-modules $X$ such that there is a path of irreducible morphisms
   %$\begin{tikzcd}
   ${X=M_0} \longrightarrow {M_1} \longrightarrow{\cdots} \longrightarrow {M_t=M}$
   %\end{tikzcd}$ 
   with all $M_i$ indecomposable.
   The \dfn{right cone} ($\longrightarrow {M}~~{\relax}$) is defined dually.

   \begin{corollary}
      Let $H$ be a representation-infinite hereditary algebra. Then $\cat[U]$ is an aisle in $\der^b(\mod H)$ without Ext-projectives and such that $\cat_1\sset \cat[U]\cap\cat[U]^\perp[1]$ if and only if one of the following two statements holds:
      \begin{itemize}
         \item [(a)]$\torsion U{U^\perp[1]}$ is a split $t$-structure with no Ext-projective objects, or
         
         \item [(b)] $H$ is wild, and each regular component $\Gamma$ of the Auslander-Reiten quiver $\Gamma(\mod H)$ contains quasisimple modules $M_\Gamma$, $N_\Gamma$ such that
         \begin{align*}
         \cat[U] &= \bigvee_{\Gamma}(\longrightarrow {M_\Gamma})\vee\left(\bigvee_{j>0}(\cat_j\vee\cat[R]_j)\right)\text{, and}\\
         \cat[U]^\perp &= \bigvee_{j<0}(\cat_j\vee\cat[R]_j)\vee \left(\bigvee_{\Gamma}(\longrightarrow {N_\Gamma}->{\relax})\right).
         \end{align*}
      \end{itemize}
   \end{corollary}
   \begin{proof}
      This follows at once from theorem~\ref{t:3.3} and \cite{AK}, theorem~(B).
   \end{proof}

   \section{Piecewise hereditariness}

   We now start our study of split torsion pairs / $t$-structures.
   Our objective in this section is to prove that the mere existence of a split $t$-structure with nontrivial heart in the bounded derived category of a finite dimensional algebra $A$ suffices to imply that the module category of such an algebra is derived equivalent to a hereditary category $\cat[H]$, so that we only need to study the split $t$-structures in the derived category $\der^b(\cat[H])$.

   We recall that an algebra $A$ is \dfn{piecewise hereditary} if $\mod A$ is derived equivalent to an hereditary category $\cat[H]$, see~\cite{HRS2}.
   Typical example of piecewise hereditary algebras are the quasitilted algebras of~\cite{HRS1} and the iterated tilted algebras of~\cite{H1}.

   \begin{lemma}\label{l:4.1}
      Let $\cat[K]$ be a Krull-Schmidt triangulated category, and $\torsion U{U^\perp[1]}$ be a split $t$-structure in $\cat[K]$.
      Then $\cat[U]$ is triangulated if and only if the heart $\cat[U]\cap\cat[U]^\perp[1]$ is zero.
   \end{lemma}
   \begin{proof}
      Indeed, $\cat[U]$ is not triangulated if and only if there exists an indecomposable object $X\in\cat[U]$ such that $X[-1]\notin\cat[U]$.
      Because $\torsion U{U^\perp[1]}$ is split, $X[-1]\notin\cat[U]$ means that $X[-1]\in\cat[U]^\perp$ or, equivalently, $X\in\cat[U]^\perp[1]$.
      Then $\cat[U]$ is not triangulated if and only if there exists an indecomposable object in the heart $\cat[U]\cap\cat[U]^\perp[1]$.
   \end{proof}

   Before quoting our next result, we need some terminology.
   Let $\cat[K]$ be a Krull-Schmidt triangulated category and $X$, $Y$ be two indecomposable objects in $\cat[K]$.
   A \dfn{semipath} from $X$ to $Y$ (path in the terminology of~\cite{R2}) is a sequence $(X=X_0,X_1,\dotsc,X_n=Y)$
   of indecomposable objects $X_i$ in $\cat[K]$ such that, for each $i$, we have $\Hom_{\cat[K]}(X_i,X_{i+1})\ne 0$ or $X_{i+1}=X_i[1]$.
   In the latter case, we say that a \dfn{jump} occurs.
   Thus, a semipath without jumps is a path in the sense of~\cite{HRS1}.
   A \dfn{semiwalk} from $X$ to $Y$ is a sequence $(X=X_0,X_1,\dotsc,X_n=Y)$ of indecomposable objects $X_i$ in $\cat[K]$ such that, for each $i$, one of the following three conditions occurs: $\Hom_{\cat[K]}(X_i,X_{i+1})\ne 0$, $\Hom_{\cat[K]}(X_{i+1},X_{i})\ne 0$ or $X_{i+1}=X_i[s]$ for some $s\in\ZZ$.
   The category $\cat[K]$ is called \dfn{semiconnected} (path-connected in the terminology of~\cite{R2}) if, given any two indecomposable objects $X$, $Y$ in $\cat[K]$, there exists a semiwalk from $X$ to $Y$.
   It is shown in~\cite{R2} that, if $\cat[K]$ is a semiconnected Krull-Schmidt triangulated category, then $\cat[K]$ is the derived category of a hereditary category if and only if there exists an indecomposable object $X$ in $\cat[K]$ with no semipath from $X[1]$ to $X$.

   \begin{lemma}\label{l:4.2}
      Let $\cat[K]$ be a Krull-Schmidt triangulated category, and $\torsion U{U^\perp[1]}$ a split $t$-structure in $\cat[K]$.
      Then $\cat[K]$ contains no semipath from an indecomposable object in $\cat[U]$ to one in $\cat[U]^\perp$.
   \end{lemma}
   \begin{proof}
      Assume that $\cat[K]$ contains a semipath $(X=X_0,X_1,\dotsc,X_n=Y)$ from $X\in\cat[U]$ to $Y\in\cat[U]^\perp$.
      We first claim that this semipath contains no jumps.
      For, assume this is the case. Then $\cat[K]$ contains a semipath $(Y_0,Y_1,\dotsc,Y_m)$ from $Y_0\in\cat[U]$ to $Y_m\in\cat[U]^\perp$
      containing a minimal number of jumps.
      Assume that the first jump occurs at $i$, so that $Y_{i+1}=Y_i[1]$.
      The semipath $(Y_0,Y_1,\dotsc,Y_{i-1},Y_i=Y_{i+1}[-1],Y_{i+2}[-1],\dotsc,Y_m[-1])$ contains one jump less than $(Y_0,Y_1,\dotsc,Y_m)$.
      Because $Y_m\in\cat[U]^\perp$, we have $Y_m[-1]\in\cat[U]^\perp$ as well.
      On the other hand, $Y_0\in\cat[U]$.
      We thus get a contradiction to our minimality hypothesis.
      This establishes our claim.

      Thus, our semipath $(X=X_0,X_1,\dotsc,X_n=Y)$ is a path and $\Hom_{\cat[K]}(X_i,X_{i+1})\ne 0$ for all $i$.
      Because $\Hom_{\cat[K]}(X,X_1)\ne 0$ and $X\in\cat[U]$, we get $X_1\notin\cat[U]^\perp$.
      The $t$-structure being split, we get $X_1\in\cat[U]$.
      Inductively, we get $Y=X_n\in\cat[U]$.
      But then $Y\in\cat[U]\cap\cat[U]^\perp=0$, a contradiction.
   \end{proof}

   We now prove the main result of this section.
   \begin{thm}
      Let $\cat[K]$ be a semiconnected Krull-Schmidt triangulated category, and \linebreak[4]$\torsion U{U^\perp[1]}$ be a split $t$-structure in $\cat[K]$ with nonzero heart.
      Then there exists a hereditary category $\cat[H]$ such that $\cat[K]\cong \der^b(\cat[H])$.
   \end{thm}
   \begin{proof}
      Let $X$ be an indecomposable object in the heart.
      We claim that there is no semipath from $X[1]$ to $X$.
      Indeed, assume that such a semipath  $(X[1]=X_0,X_1,\dotsc,X_n=X)$ exists.
      Then there exists another semipath  $(X=X_0[-1],X_1[-1],\dotsc,X_n[-1]=X[-1])$ from $X$ to $X[-1]$.
      Because $X$ lies in the heart, we have $X\in\cat[U]$.
      But also $X\in\cat[U]^\perp[1]$ which implies $X[-1]\in\cat[U]^\perp$ and then lemma~\ref{l:4.2}
      gives a contradiction.
      This proves our claim.
      Invoking Ringel's result as quoted above completes the proof.
   \end{proof}

   \begin{corollary}\label{c:4.4}
      Let $A$ be a finite dimensional connected algebra, and $\torsion U{U^\perp[1]}$ be a split $t$-structure with nontrivial heart
      in $\der^b(\mod A)$.
      Then $A$ is piecewise hereditary.\qed
   \end{corollary}

   \section{Tilting and torsion pairs}

   Let $\cat[H]$ be a hereditary category, with tilting object $T$.
   The endomorphism algebra $A=\End_{\cat[H]} T$ is then said to be \dfn{quasitilted}, see~\cite{HRS1}.
   Typical examples of quasitilted algebras are the tilted algebras, see~\cite{ASS} or~\cite{H1},
   and the canonical algebras, see~\cite{R1}.
   The tilting object $T$ induces a torsion pair $\torsion {T(T)}{F(T)}$ in $\cat[H]$ and a split torsion pair $\torsion {X(T)}{Y(T)}$ in $\mod A$
   by $\cat[T](T)=\set{X\in\cat[H]\mid \Ext^1_{\cat[H]}(T,X)=0}$, $\cat[F](T)=\set{Y\in\cat[H]\mid\Hom_{\cat[H]}(T,Y)=0}$
   and $\cat[X](T)=\Ima \Ext^1_{\cat[H]}(T,-)$, $\cat[Y](T)=\Ima\Hom_{\cat[H]}(T,-)$.
   Considering these subcategories as embedded in $\der^b(\cat[H])$, we have $\cat[Y](T)=\cat[T](T)$ and $\cat[X](T)=\cat[F](T)[1]$.
   We first prove that any split torsion pair in $\cat[H]$ induces a split torsion pair in $\mod A$.

   \begin{lemma}
      Let $\cat[H]$ be an hereditary category with tilting object $T$ and $A=\End_{\cat[H]}T$.
      A split torsion pair $\torsion TF$ in $\cat[H]$ induces a split torsion pair $\torsion{T'}{F'}$ in $\mod A$.
   \end{lemma}
   \begin{proof}
      Let $\torsion TF$ be a split torsion pair in $\cat[H]$.
      We claim that $\cat[T]'=\left(\cat[Y](T)\cap\cat[T]\right)\vee\cat[X](T)$ is a torsion class in $\mod A$.

      We first prove that $\cat[T]'$ is closed under quotients.
      Let $X \longrightarrow Y$ be an epimorphism in $\mod A$ with $X\in\cat[T]'$.
      We may assume that $X$ is indecomposable.
      If $X\in\cat[X](T)$, then $Y\in\cat[X](T)$ because $\cat[X](T)$ is a torsion class.
      Therefore, in this case, $Y\in\cat[T]'$.
      Otherwise, $X\in\cat[Y](T)\cap\cat[T]$.
      Because $X\in\cat[T]$, it is an object in $\cat[H]$, hence so is $Y$ and then $Y\in\cat[T]$.
      But also, $X\in\cat[Y](T)\cap\cat[H]=\cat[T](T)$ gives $Y\in\cat[T](T)$, because $\cat[T](T)$ is a torsion class in $\cat[H]$.
      But then $Y\in\cat[Y](T)\cap\cat[T]=\cat[T]'$.

      We next prove that $\cat[T]'$ is closed under extensions.
      Let $0 \longrightarrow X \longrightarrow Y \longrightarrow Z \longrightarrow 0$  be a short exact sequence in $\mod A$, with $X$, $Z\in\cat[T]'$.
      We may assume that both $X$ and $Z$ are indecomposable.
      If $X$ and $Z$ both belong to $\cat[X](T)$ or both belong to $\cat[Y](T)\cap\cat[T]$,
      then so does $Y$ because each of these classes is closed under extensions.
      Because $\torsion {X(T)}{Y(T)}$ is split, the only case to consider is when $X\in\cat[Y](T)\cap\cat[T]$ and $Z\in\cat[X](T)$.
      Using again that $\torsion {X(T)}{Y(T)}$ is split we have $Y=Y'\oplus Y''$ with $Y'\in\cat[X](T)$, $Y''\in\cat[Y](T)$.
      It suffices to prove that $Y''\in\cat[T]$.
      Now $Y''\in\cat[Y](T)$ implies $Y''\in\cat[H]$.
      Then, either the short exact sequence above splits, and we are done, or else there exists a nonzero morphism $X \longrightarrow b{Y''}$ in $\cat[H]$.
      Because $X\in\cat[T]$, no indecomposable summand of $Y''$ belongs to $\cat[F]$.
      But $\torsion TF$ splits in $\mod A$, therefore $Y''\in\cat[T]$.
      This establishes our claim.

      Let $\cat[F]'=\cat[T']^\perp$.
      In order to prove that $\torsion {T'}{F'}$ is split, it suffices to prove that $\cat[F]'=\cat[Y](T)\setminus\cat[T]=\cat[Y](T)\cap\cat[F]$.
      Assume that $X\in\cat[Y](T)\setminus\cat[T]$, we claim that $\Hom_A(-,X)\big|_{\cat[T]'}=0$.
      Indeed, $X\in\cat[F]$ implies that $\Hom_A(-,X)\big|_{\cat[T]}=0$, hence
      $\Hom_A(-,X)\big|_{\cat[T]\cap\cat[Y](T)}=0$.
      But also $X\in\cat[Y](T)$ implies $\Hom_A(-,X)\big|_{\cat[X](T)}=0$.
      Therefore $\Hom_A(-,X)\big|_{\cat[T]'}=0$, as required.
      Conversely, let $X\in\cat[F]'$ be indecomposable.
      Then $X\notin\cat[T]'$.
      In particular, $X\notin\cat[X](T)$.
      Therefore $X\in\cat[Y](T)$ because $\torsion {X(T)}{Y(T)}$ is split.
      But then $X\notin\cat[T]'$ also implies that $X\notin\cat[T]$.
      Therefore $X\in\cat[Y](T)\setminus\cat[T]$.
      The proof is now complete.
   \end{proof}

   Observe that nontrivial torsion classes in $\cat[H]$ map to nontrivial torsion classes in $\mod A$. Indeed, $\cat[T]\cap\cat[T](T)\ne 0$ above implies $\cat[T]'\cap\cat[Y](T)\ne 0$.

   Let $H$ be a hereditary algebra.
   We recall that an algebra $A$ is \dfn{tilted of type $H$} if there exists a tilting $H$-module $T$ such that $A=\End T_H$, see~\cite{ASS}.
   We denote by $\cat[P]_A$, $\cat[I]_A$ respectively the postprojective and the preinjective components of the Auslander-Reiten quiver $\Gamma(\mod A)$,
   and by $\cat[P]_H$, $\cat[I]_H$ those of $\Gamma(\mod H)$.

   \begin{proposition}\label{p:5.2}
      Let $A$ be a representation-infinite tilted algebra of type $H$ having a complete slice in the preinjective component.
      Then there exists a bijective correspondence between the class of split torsion pairs $\torsion TF$ in $\mod A$ such that $\cat[P]_A\sset\cat[F]$, $\cat[I]_A\sset\cat[T]$ and the class of split torsion pairs $\torsion{T'}{F'}$ in $\mod H$ such that  $\cat[P]_H\sset\cat[F]'$, $\cat[I]_H\sset\cat[T]'$.
   \end{proposition}
   \begin{proof}
      There exists a tilting module $T$ such that $A=\End T$.
      The correspondence between $\mod A$ and $\mod H$ induced by the
      tilting functors is summarised in the following picture (see~\cite{ASS}).

%s\vskip 4cm 

\begin{figure}[h!]\hspace{-8cm}

%\begin{center}
%\center{
 \includegraphics[scale=0.8]{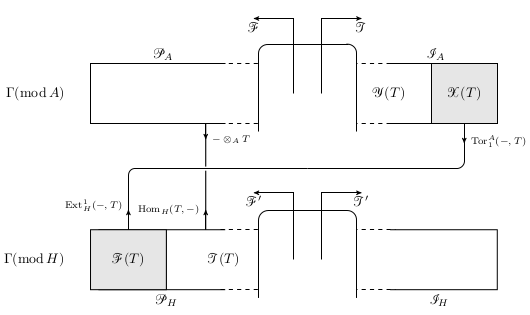}

%\end{center}

\end{figure}

      Note that, while $\torsion {X(T)}{Y(T)}$ is split in $\mod A$, $\torsion {T(T)}{F(T)}$ is usually not split in $\mod H$.
      The proof is done in three steps.
      
          We start by defining a map $\zeta$ from the set of split torsion classes $\cat[T]$ in $\mod A$ with $\cat[I]_A\sset\cat[T]$, $\cat[P]_A\sset\cat[T]^\perp=\cat[F]$ to the set of split torsion classes $\cat[T]'$ in $\mod H$ with $\cat[I]_H\sset\cat[T]'$, $\cat[P]_A\sset\cat[T]'^\perp=\cat[F]'$.

         Let $\cat[T]$ be a split torsion class in $\mod A$ and let
         \[\cat[T]'=\Ima(\cat[T]\otimes_A T)=\set{M\otimes_A T\mid M\in\cat[T]} \]
         in $\mod H$. Thus $\cat[T]'$ is actually contained inside $\cat[T](T)$.

         We claim that, in fact, $\cat[T]'=\Ima\left((\cat[Y](T)\cap\cat[T])\otimes_AT\right)$.
         Indeed, let $M\in\cat[T]$ and consider its canonical sequence in the torsion pair $\torsion {X(T)}{Y(T)}$
         $$0 \longrightarrow {M_{\cat[X]}} \longrightarrow M \longrightarrow {M_{\cat[Y]}} \longrightarrow 0$$ 
         with $M_{\cat[X]}\in\cat[X](T)$, $M_{\cat[Y]}\in\cat[Y](T)$.
         Applying $-\otimes_A T$, we get $M\otimes_AT\cong M_{\cat[Y]}\otimes_A T$,
         because $M_{\cat[X]}\otimes_AT=0$.
         Moreover, $M\in\cat[T]$ implies $M_{\cat[Y]}\in\cat[T]$ hence $M_{\cat[Y]}\in\cat[Y](T)\cap\cat[T]$.
         This establishes our claim

         We next claim that $\cat[P]_H\cap\cat[T]'=0$.
         Indeed, let $X\in\cat[P]_H\cap\cat[T]'$ be indecomposable.
         Because $X\in\cat[T]'$, there exists $M\in\cat[Y](T)\cap \cat[T]$ such that $X\cong M\otimes_A T$.
         Because $M\in\cat[Y](T)$, we have $\Hom_H(T,X)\cong\Hom_H(T,M\otimes_A T)\cong M\in\cat[T]$.
         On the other hand, $X\in\cat[P]_H$ implies $\Hom_H(T,X)\in\cat[P]_A$.
         This contradicts the fact that $\cat[P]_A\sset\cat[F]$ by hypothesis.
         Our claim is proved.

         We now prove that $\cat[T]'$ is a split torsion class by proving that it is closed under successors.
         Assume we have a nonzero morphism $ X \longrightarrow Y$ with $X$, $Y$ indecomposable and $X\in\cat[T]'$.
         We first show that, under these hypotheses, $Y\in\cat[T](T)$.
         Consider the canonical sequence of $Y$ in the torsion pair $\torsion {T(T)}{F(T)}$
         \[0 \longrightarrow {Y_{\cat[T]}}\longrightarrow Y \longrightarrow {Y_{\cat[F]}} \longrightarrow 0 \]
         with $Y_{\cat[T]}\in\cat[T](T)$, $Y_{\cat[F]}\in\cat[F](T)$.
         Assume $Y_{\cat[F]}\ne 0$.
         Because $Y_{\cat[F]}\in\cat[F](T)\sset\add\cat[P]_H$, every indecomposable summand of $Y_{\cat[F]}$ lies in $\cat[P]_H$.
         Because $\cat[P]_H$ is closed under predecessors, $Y$ and also $X$ are in $\cat[P]_H$.
         But this contradicts the facts that $X\in\cat[T]'$ and $\cat[P]_H\cap\cat[T]'=0$.
         Therefore $Y_{\cat[F]}=0$ and $Y=Y_{\cat[T]}\in\cat[T](T)$, as required.
         Hence $\Hom_H(T,Y)\in\cat[Y](T)$.
         Because $X\in\cat[T]'\sset\cat[T](T)$, the tilting theorem asserts the existence of a nonzero morphism ${\Hom_H(T,X)} \longrightarrow {\Hom_H(T,Y)}$ in $\mod A$.
         Now $\Hom_H(T,X)\in\cat[T]$: indeed, $X\in\cat[T]'$ says that there exists $M\in\cat[Y](T)\cap\cat[T]$ such that $X\cong M\otimes_AT$.
         Therefore $\Hom_H(T,X)\cong\Hom_H(T,M\otimes_AT)\cong M\in\cat[T]$, where we have used that $M\in\cat[Y](T)$.
         Because $\cat[T]$ is closed under successors, we have $\Hom_H(T,Y)\in\cat[T]$.
         Because $Y\in\cat[T](T)$, we have $Y\cong\Hom_H(T,Y)\otimes_A T\in\cat[T]'$.

         Let $\cat[F]'=\cat[T]'^\perp$. Then $\torsion {T'}{F'}$ is a split torsion pair in $\mod H$.
         Moreover $\cat[T]'\cap\cat[P]_H=0$ implies $\cat[P]_H\sset\cat[F]'$ and also $\cat[I]_H=\left(\cat[I]_A\cap\cat[Y](T)\right)\otimes_AT\sset\cat[T]'$,
         because $\cat[I]_A\sset\cat[T]$.

         For future use, we characterise the modules in $\cat[F]'$.
         We have $X\in\cat[F]'$ if and only if $\Hom_H(-,X)\big|_{\cat[T]'}=0$, that is, $\Hom_H(L\otimes _AT,X)=0$ for all $L\in\cat[T]$ or, equivalently, $\Hom_A\left(L,\Hom_H(T,X)\right)=0$ for all $L\in\cat[T]$.
         Thus $X\in\cat[F]'$ if and only if $\Hom_H(T,X)\in\cat[F]$.

         This completes the definition of the map %$\begin{tikzcd}
         $\zeta:\cat[T]\longrightarrow {\cat[T]'}$%\end{tikzcd}$.

         We next define a map $\chi$ from the set of split torsion classes $\cat[T]'$ in $\mod H$ with $\cat[I]_H\sset\cat[T]'$, $\cat[P]_H\sset\cat[T]'^\perp=\cat[F]'$ to the set of split torsion classes $\cat[T]$ in $\mod A$ with $\cat[I]_A\sset\cat[T]$, $\cat[P]_A\sset\cat[T]^\perp=\cat[F]$.

         Let $\cat[T]'$ be a split torsion class in $\mod H$ and $\cat[F]'=\cat[T]'^\perp$.
         Let \[\cat[F]=\Hom_H(T,\cat[F]')=\set{\Hom_H(T,X)\mid X\in\cat[F]'}.\]
         As in 1. above, it is easy to see that, in fact, $\cat[F]=\Hom(T,\cat[F]'\cap\cat[T](T))\sset\cat[Y](T)$.

         We claim that $\cat[I]_A\cap\cat[F]=0$.
         Indeed, assume $M\in\cat[X](T)$, then $M\notin\cat[Y](T)$ hence $M\notin\cat[F]$.
         Otherwise, $M\in\cat[Y](T)\cap\cat[I]_A$ implies that $M\otimes_AT\in\cat[T](T)\cap\cat[I]_H\sset\cat[T]'$.
         If $M\in\cat[F]$, then there exists $X\in\cat[F]'\cap\cat[T](T)$ such that $M\cong\Hom_H(T,X)$.
         But then, $X\in\cat[T](T)$ yields $M\otimes_AT\cong\Hom_H(T,X)\otimes_A T\cong X\in\cat[F]'$, a contradiction. Therefore $M\notin\cat[T]$, establishing our claim.

         We prove that $\cat[F]$ is a split torsion class by proving it is closed under predecessors.
         Assume we have a nonzero morphism $\longrightarrow L->M$, with $L$, $M$ indecomposable and $M\in\cat[F]$.
         Because of our claim above, $M\notin\cat[I]_A$.
         Hence, $M\in\cat[Y](T)$.
         Because $\cat[Y](T)$ is closed under predecessors, $L\in\cat[Y](T)$ so $L\otimes_AT\in\cat[T](T)$.
         The tilting theorem yields a nonzero morphism $ {L\otimes_AT} \longrightarrow {M\otimes_AT}$.
         Because $M\in\cat[F]$, there exists $X\in\cat[F]'\cap\cat[T](T)$ such that $M\cong\Hom_H(T,X)$.
         Because $X\in\cat[T](T)$, we have $M\otimes_AT\cong\Hom_H(T,X)\otimes_AT\cong X\in\cat[F]'$.
         Because $\cat[F]'$ is closed under predecessors, $L\otimes_AT\in\cat[F]'$.
         Then $L\in\cat[Y](T)$ yields $L\cong\Hom_H(T,L\otimes_AT)\in\cat[F]$. We are done.

         Letting $\cat[T]={^\perp}\cat[F]$, we get a split torsion pair $\torsion TF$ in $\mod A$.
         Also, $\cat[I]_A\cap\cat[F]=0$ yields $\cat[I]_A\sset\cat[T]$, and $\cat[P]_A=\Hom_H(T,\cat[P]_H\cap\cat[T](T))\sset\cat[F]$, because $\cat[P]_H\sset\cat[F]'$.

         We now characterise the modules in $\cat[T]$.
         We have $L\in\cat[T]$ if and only if \linebreak[4]$\Hom_A(L,-)\big|_{\cat[F]}=0$, that is, if and only if $\Hom_A\left(L,\Hom_H(T,X)\right)=0$ for all $X\in\cat[F]'$ or, equivalently, $\Hom_H(L\otimes_AT,X)=0$ for all $X\in\cat[F]'$.
         Thus, $L\in\cat[T]$ if and only if $L\otimes_AT\in\cat[T]'$.

         This completes the definition of the map %$\begin{tikzcd}
         ${\chi:\cat[T]'}\longrightarrow \cat[T]$%\end{tikzcd}$.

         \item Finally, we prove that $\zeta$ and $\chi$ are inverse to each other.
         We first show that $\chi\circ\zeta=\id$.
         Let $\cat[T]$ be a split torsion class in $\mod A$ such that $\cat[I]_A\sset\cat[T]$, $\cat[P]_A\sset\cat[T]^\perp$.
         Let $L\in\cat[T]$, then $L\otimes_AT\in\Ima(\cat[T]\otimes_AT)=\zeta(\cat[T])$.
         Therefore $\cat[T]\sset\chi\zeta(\cat[T])$.

         Conversely, let $L\in\chi\zeta(\cat[T])$.
         Then $L\otimes_AT\in\zeta(\cat[T])$ and there exists $L'\in\cat[T]\cap\cat[Y](T)$ such that $L\otimes_AT\cong L'\otimes_AT$.
         Denoting by $\delta_L$ the unit of the $\otimes-\Hom-$adjunction, we have $\delta_L: L \longrightarrow {\Hom_H(T,L\otimes_AT)}\cong \Hom_H(T,L'\otimes_AT)\cong L'$ because $L'\in\cat[Y](T)$.
         Now $\cat[Y](T)$ is closed under successors, hence $L\in\cat[Y](T)$ and so $\delta_L$ is an isomorphism.
         Thus $L\cong L'\in\cat[T]$.
         Therefore $\chi\zeta(\cat[T])\sset\cat[T]$ and we have proven that $\chi\circ\zeta=\id$.

         In order to prove that $\zeta\circ\chi=\id$, let $\cat[T]'$ be a split torsion class in $\mod H$ such that $\cat[I]_H\sset\cat[T]'$, $\cat[P]_H\sset\cat[T]'^\perp$.
         Let $X\in\zeta\chi(\cat[T]')$.
         Then there exists $L\in\chi(\cat[T]')$ such that $X\cong L\otimes_A T$.
         But $L\in\chi(\cat[T]')$ implies $L\otimes_AT\in\cat[T]'$.
         Therefore $X\in\cat[T]'$ and so $\zeta\chi(\cat[T]')\sset\cat[T]'$.

         Conversely, let $X\in\cat[T]'$. Because $\cat[T]'\sset\cat[T](T)$, there exists $L\in\cat[Y](T)$ such that $X\cong L\otimes_AT$.
         Because $L\otimes_AT\in\cat[T]'$, we have $L\in\chi(\cat[T]')$ so $L\in\chi(\cat[T]')\cap\cat[Y](T)$ and then $X\in\Ima\left((\chi(\cat[T]')\cap\cat[Y](T)\otimes_AT\right)=\zeta\chi(\cat[T]')$.
         Thus $\cat[T]'\sset\zeta\chi(\cat[T]')$ and so $\zeta\circ\chi=\id$. \qedhere
      
   \end{proof}

This leads us to our main result of this section.

\begin{thm}\label{t:5.3}
   Let $A$ be a representation-infinite tilted algebra of type $H$ having a complete slice in the preinjective component.
   Then there are bijective correspondences between the following three classes:
   \begin{itemize}
      \item [(a)] Split torsion pairs $\torsion TF$ in $\mod A$ such that $\cat[I]_A\sset\cat[T]$, $\cat[P]_A\sset\cat[F]$.
          
      \item [(b)] Split torsion pairs $\torsion {T'}{F'}$ in $\mod H$ such that $\cat[I]_H\sset\cat[T]'$, $\cat[P]_H\sset\cat[F]'$.
          
      \item [(c)] Split $t$-structures $\torsion U{U^\perp[1]}$ in $\der^b(\mod H)$ such that $\cat_1\sset\cat[U]\cap\cat[U]^\perp[1]$.
   \end{itemize}
\end{thm}
\begin{proof}
   We combine proposition~\ref{p:5.2}, theorem~\ref{t:3.3} and the remark just following it.
\end{proof}

We shall give a precise description of the $t$-structures considered above in section~\ref{sec:6} below.
Note that, if $A$ is a representation-infinite tilted algebra of euclidean type, then, up to duality, we may assume that it has a complete slice in the preinjective component.
The above theorem then applies.

\section{Split $t$-structures}\label{sec:6}

The objective of this final section is to give a complete description of the split $t$-structures in $\der^b(\mod H)$ when $H$ is an hereditary algebra.
We start by considering the case where the aisle of the $t$-structure admits an indecomposable Ext-projective object.
For the notion of presection, we refer the reader to~\cite{ABS}.

\begin{lemma}\label{l:6.1}
   Let $Q$ be a quiver and $\torsion U{U^\perp[1]}$ a split $t$-structure in $\der^b(\mkQ)$. If a component $\Gamma$ of $\Gamma(\der^b(\mkQ))$ contains an indecomposable Ext-projective in $\cat[U]$, then:
   \begin{itemize}
      \item [(a)]$\Gamma$ is a transjective component in $\Gamma(\der^b(\mkQ))$.
      
      \item [(b)] The indecomposable Ext-projectives in $\cat[U]$ form a section in $\Gamma$.
      
      \item [(c)] There are no indecomposable Ext-projectives in $\cat[U]$ in the other components of \linebreak[4] $\Gamma(\der^b(\mkQ))$.
   \end{itemize}
\end{lemma}
\begin{proof}
   Let $E_0\in\cat[U]$ be an indecomposable Ext-projective in $\cat[U]$ lying in $\Gamma$. Then $\tau E_0\in\cat[U]^\perp$.
   \begin{itemize}
      \item [(a)] Assume first that $\Gamma$ is a stable tube.
      Then there exists $s\ge 1$ such that $E_0=\tau^s E_0$.
      Because $s\ge 1$, $\tau^sE_0$ precedes $\tau E_0$ and hence lies in $\cat[U]^\perp$, because of lemma~\ref{l:4.2}.
      But now $E_0\in\cat[U]$, and we have a contradiction.
      If $\Gamma$ is a component of type $\ZZ\AA_\infty$, there exist $t\ge 1$ and a nonzero morphism ${E_0} \longrightarrow {\tau^t E_0}$, see~\cite{K}(1.3).
      Again, $\tau^t E_0$ precedes $\tau E_0$ and hence lies in $\cat[U]^\perp$.
      Then $E_0\in\cat[U]$ yields the same contradiction as before.
      Therefore $E_0$ lies neither in a stable tube, nor in a component of type $\ZZ\AA_\infty$.
      Hence, $\Gamma$ is a transjective component.
      
      \item  [(b)]Because $\Gamma$ is transjective, it is of the form $\ZZ Q$.
      In order to prove that the Ext-projectives constitute a section in $\Gamma$, it suffices to prove that they form a presection, because of~\cite{ABS}~proposition~7.
      Let $ {E_0} \longrightarrow X$ be an arrow in $\Gamma$, with $E_0$ indecomposable Ext-projective in $\cat[U]$.
      Observe that, because $X$ succedes $E_0$, we have $X\in\cat[U]$.
      Assume $X$ is not Ext-projective.
      Then $\tau X\notin\cat[U]^\perp$.
      Because $\torsion U{U^\perp[1]}$ is split, we get $\tau X\in\cat[U]$.
      On the other hand, there is an arrow $ {\tau^2 X} \Longrightarrow {\tau E_0}$ and $\tau E_0\in\cat[U]^\perp$.
      Therefore $\tau^2X\in\cat[U]^\perp$.
      This implies that $\tau X$ is Ext-projective.
      Dually, if $ Y \longrightarrow {E_0}$ is an arrow in $\Gamma$, then either $Y$ or $\tau^{-1}Y$ is Ext-projective in $\cat[U]$. This completes the proof.

      \item [(c)] It follows from (b) that the number of isomorphism classes of indecomposable Ext-projectives in $\cat[U]$ lying in $\Gamma$ equal $|Q_0|=\rk K_0(\kk Q)$.
      Because of~\cite{AST}~theorem~2.3, there are no other Ext-projectives.\qedhere
   \end{itemize}
\end{proof}

\begin{corollary}\label{c:6.2}
   Let $Q$ be a quiver and $\torsion U{U^\perp[1]}$ be a split $t$-structure in $\der^b(\mkQ)$.
   Then the number of isomorphism classes of indecomposable Ext-projectives in $\cat[U]$ is either equal to zero or to $|Q_0|$.\qed
\end{corollary}

We are now able to state and prove our main result of this section, which describes completely the split $t$-structures considered in corollary~\ref{c:3.2} and theorem~\ref{t:5.3}.

\begin{thm}
Let $Q$ be a quiver and $\torsion U{U^\perp[1]}$ be a split $t$-structure in $\der^b(\mkQ)$. Then we have one of the following:
\begin{itemize}
   \item [(a)] If $\cat[U]$ admits at least one indecomposable Ext-projective, then it admits $|Q_0|$, the set of which forms a section in a transjective component $\cat_i$ and then \[\cat[U]=(\Succ\Sigma)\vee\cat[R]_i\vee\left(\bigvee_{j>i}(\cat_j\vee\cat[R]_j)\right).\]
       
   \item [(b)] If $\cat[U]$ has no Ext-projective and $\kk Q$ is tame, then there exist $i\in\ZZ$ and a subset $L\sset\dd P_1(\kk)$ such that
   \[\cat[U]=\left(\bigvee_{\lambda\in L}\cat[T]_\lambda\right)\vee\left(\bigvee_{j>i}(\cat_j\vee\cat[R]_j)\right)\]
   where $\cat[R]_i=\left(\cat[T]_\lambda\right)_{\lambda\in\dd P_1(\kk)}$.
   
   \item [(c)] If $\cat[U]$ has no Ext-projective and $\kk Q$ is wild, then there exists $i$ such that either
      \[ \cat[U]=\bigvee_{j>i}(\cat_j\vee\cat[R]_j)\qquad\text{or}\qquad \cat[U]=\cat[R]_i\vee\left(\bigvee_{j>i}(\cat_j\vee\cat[R]_j)\right).\]
\end{itemize}
\end{thm}
\begin{proof}
   Assume first that $\kk Q$ is representation-finite.
   In this case, either $\cat[U]$ is triangulated or else there exists an indecomposable object $X\in\cat[U]$ such that $X[-1]\notin\cat[U]$.
   Hence there exists an indecomposable object $E_0$ in the $\tau$-orbit of $X$ such that $E_0\in\cat[U]$ but $\tau E_0\notin\cat[U]$.
   Because $\torsion U{U^\perp[1]}$ is split, $\tau E_0\in\cat[U]^\perp$ and $E_0$ is Ext-projective. Lemma~\ref{l:6.1} then gives a section $\Sigma$ in $\Gamma(\der^b(\mkQ))$ consisting of Ext-projectives. It is then easily seen that $\cat[U]=\Succ\Sigma$.

   Thus, assume that $\kk Q$ is representation-infinite.

   Assume first that $\kk Q$ is wild.
   In this case, the transjective components $\cat_i$ are of the form $\ZZ Q$, while the regular families $\cat[R]_i$ consist each of infinitely many components of type $\ZZ\AA_\infty$.

   In case $\cat[U]$ admits an indecomposable Ext-projective, then, because of lemma~\ref{l:6.1}, this Ext-projective lies in some $\cat_i$, there is a section in $\cat_i$ consisting of Ext-projectives and we conclude as in the representation-finite case.
   We may thus assume that $\cat[U]$ has no Ext-projectives.

   Let $X$, $Y$ be any two indecomposable regular $\kk Q$-modules.
   Because of~\cite{K}(1.3), there exists $t>0$ such that $\Hom_H(X,\tau^tY)\ne 0$.
   Thus, if $X\in\cat[U]$ then so does $\tau^tY$ and hence so does $Y$.
   Dually, if $Y\in\cat[U]^\perp$, then $X\in\cat[U]^\perp$.
   This proves that either all regular components in a given $\cat[R]_i$ lie in $\cat[U]$, or they all lie in $\cat[U]^\perp$.
   Thus we have one of the following cases:
      either there exists $i\in\ZZ$ such that $\cat[R]_i\sset\cat[U]=0$ and $\cat_{i+1}\sset\cat[U]$ and then $\cat[U]=\bigvee_{j>i}(\cat_j\vee\cat[R]_j)$,
     or else there exists $i\in\ZZ$ such that $\cat_{i-1}\cap\cat[U]=0$ and $\cat[R]_i\sset\cat[U]$,
      in which case we have $\cat[U]=\cat[R]_i\vee\left(\bigvee_{j>i}(\cat_j\vee\cat[R]_j)\right)$.

   Finally, assume that $\kk Q$ is tame. Again the $\cat_i$ are of the form $\ZZ Q$ while each regular family $\cat[R]_i$ consists of a separating family of pairwise orthogonal stable tubes indexed by the projective line $\dd P_1(\kk)$.
   If $\cat[U]$ admits an indecomposable Ext-projective, then we proceed as in the wild case above.
   If not, then there are two cases.
   If there exists $i\in\ZZ$ with $\cat_i\cap\cat[U]=0$ and $\cat[R]_i\cap\cat[U]\ne 0$, let $\cat[T]_\lambda$ be a tube in $\cat[R]_i$ such that $\cat[T]_\lambda\cap\cat[U]\ne 0$, then $\cat[T]_\lambda\sset\cat[U]$.
   If, on the other hand, $\cat[T]_\mu\cap\cat[U]=0$, then $\cat[T]_\mu\sset\cat[U]^\perp$.
   The pairwise orthogonality of the tubes implies the existence of a subset $L\sset\dd P_1(\kk)$ such that
   \[ \cat[U]=\left(\bigvee_{\lambda\in L}\cat[T]_\lambda\right)\vee\left(\bigvee_{j>i}(\cat_j\vee\cat[R]_j)\right).\]
   If, on the other hand, there exists $i\in\ZZ$ such that $\cat[R]_i\cap\cat[U]=0$ and $\cat_{i+1}\cap\cat[U]\ne 0$,
      then we proceed as before taking $L=\varnothing$ and we get $\cat[U]=\bigvee_{j>i}(\cat_j\vee\cat[R]_j)$.
\end{proof}

For the notion of tilting complex, we refer the reader to~\cite{Ri}.

\begin{corollary}
   Let $\torsion U{U^\perp[1]}$ be a split $t$-structure in $\der^b(\mkQ)$ and $E_1,\dotsc,E_n$ be a complete set of representative of the isomorphism classes of indecomposable Ext-projectives in $\cat[U]$.
   Let $E=\bigoplus_{i=1}^n E_i$.
   Then
   \begin{itemize}
      \item [(a)] $E$ belongs to the heart, so $\cat[U]$ is not triangulated.
      
      \item [(b)] $E$ is a tilting complex in $\der^b(\mkQ)$ and $\cat[U]$ is the smallest suspended subcategory of $\der^b(\mkQ)$ containing $E$.
   \end{itemize}
\end{corollary}
\begin{proof}
   \begin{itemize}
      \item [(a)] We claim that, for any $i$, $E_i[-1]\notin\cat[U]$.
      Indeed, if this were the case and $E_i[-1]\in\cat[U]$, then we get $\Hom_{\der^b(\mkQ)}(E_i,E_i)=\Hom_{\der^b(\mkQ)}(E_i,E_i[-1][1])=0$ because $E_i$ is Ext-projective in $\cat[U]$, and this is an absurdity.
      This shows our claim.
      Because $\torsion U{U^\perp[1]}$ is split, $E_i[-1]\in\cat[U]^\perp$ and so $E_i\in\cat[U]^\perp[1]$.
      Because $E_i\in\cat[U]$, we indeed get $E_i\in\cat[U]\cap\cat[U]^\perp[1]$.
      Finally, $E$ lies in the heart, because each $E_i$ does.
      The last statement follows from lemma~\ref{l:4.1}.
      
      \item [(b)] Because of corollary~\ref{c:6.2}, we have $n=|Q_0|$.
      Applying~\cite{AST}~corollary~4.4, we get that $E$ is a generator of $\der^b(\mkQ)$.
      Hence it is a tilting complex.
      The second statement also follows from~\cite{AST}~corollary~4.4.\qedhere
   \end{itemize}
\end{proof}

\textsc{Acknowledgements}.  The first author gratefully acknowledges partial support from the NSERC of Canada, the FRQ-NT of Qu\'ebec and the Universit\'e de Sherbrooke.
The third author is a researcher of the CONICET (Argentina).

%\bibliographystyle{alphanum}
%
%\bibliography{references}
\end{document}